\documentclass{article}

\usepackage{arxiv}
\usepackage[utf8]{inputenc} 
\usepackage[T1]{fontenc}    
\usepackage{hyperref}       
\usepackage{url}            
\usepackage{booktabs}       
\usepackage{amsfonts}       
\usepackage{nicefrac}       
\usepackage{microtype}      
\usepackage{lipsum}		
\usepackage[english]{babel}
\usepackage{graphicx}
\usepackage{doi}
\usepackage{amsmath}
\usepackage{amsthm}
\usepackage{algorithm}
\usepackage{algpseudocode}
\usepackage[toc,page]{appendix}
\usepackage{thmtools}
\usepackage{thm-restate}

\usepackage{cleveref}
\usepackage{pst-plot}
\usepackage{subcaption}

\newtheorem{theorem}{Theorem}[section]

\newtheorem{assumption}{Assumption}
\newtheorem{lemma}{Lemma}

\theoremstyle{definition}
\newtheorem{definition}{Definition}

\DeclareMathOperator{\spanop}{span}

\DeclareMathOperator{\Ima}{Im}

\usepackage[colorinlistoftodos,bordercolor=blue,backgroundcolor=blue!20,linecolor=blue,textsize=scriptsize]{todonotes}

\title{Average-case optimization analysis for distributed consensus algorithms on regular graphs}

\author{
  Nhat Trung Nguyen \\
  MIPT\\
  \texttt{ngnhtrg@phystech.edu} \\
   \And
 Alexander Rogozin \\
  MIPT, Innopolis University \\
  \texttt{aleksandr.rogozin@phystech.edu} \\
  \And
 Alexander Gasnikov \\
  Innopolis University, MIPT, IITP \\
  \texttt{gasnikov@yandex.ru} \\
}

\begin{document}

\maketitle

\begin{abstract}
The consensus problem in distributed computing involves a network of agents aiming to compute the average of their initial vectors through local communication, represented by an undirected graph.  This paper focuses on the studying of this problem using an average-case analysis approach, particularly over regular graphs. Traditional algorithms for solving the consensus problem often rely on worst-case performance evaluation scenarios, which may not reflect typical performance in real-world applications. Instead, we apply average-case analysis, focusing on the expected spectral distribution of eigenvalues to obtain a more realistic view of performance. Key contributions include deriving the optimal method for consensus on regular graphs, showing its relation to the Heavy Ball method, analyzing its asymptotic convergence rate, and comparing it to various first-order methods through numerical experiments.
\end{abstract}



\section{Introduction}
The averaging, or consensus, problem in a connected network is a fundamental concept in distributed computing. This problem involves a network of agents connected by undirected communication links. Each agent $i \in V = \{1, \dots, n\}$ in the network holds an initial vector $x_0^{(i)}$, and the communication network is represented as a connected graph $G = (V, E)$. The objective of the agents is to compute the average of these initial vectors across the network while being restricted to communicating only with their immediate neighbors. 

Numerous algorithms have been proposed to address this problem and its variants, each offering different approaches to improve the efficiency of computing the average in a distributed network under certain conditions \cite{boyd2006randomized, berthier2020accelerated}. The main goal is to develop an iterative communication algorithm that allows each agent to efficiently determine the average of the initial vectors across the entire network. This challenge is motivated by its relevance to various fields, including decentralized optimization \cite{scaman2017optimal, rogozin2023decentralized}, distributed control and sensing \cite{boyd2006randomized}, and large-scale machine learning \cite{georgopoulos2014distributed, tsianos2012consensus}.

\subsection{Average-case optimization analysis}
The consensus problem can be formulated as a quadratic programming problem in the context of optimization. This formulation enables the development of accelerated algorithms based on optimization theory specific to quadratic programming. Traditionally, optimization algorithms are evaluated using worst-case scenarios \cite{nemirovski1995information, nesterov2004}, which involves analyzing the algorithm's complexity bound for \textit{any} input within a function class. While this approach gives us a guarantee of convergence speed, it does not always accurately reflect the algorithm's practical performance, as the most challenging cases may rarely occur in real-world applications.

To better understand the typical behavior of an algorithm, average-case analysis is employed. This approach evaluates the average performance of the algorithm across all possible inputs. Average-case analysis is commonly applied in areas like sorting \cite{knuth1997art, 10.1093/comjnl/5.1.10}, simplex method in linear programming \cite{smale1983average, vershik1983estimation, spielman2004smoothed}, and cryptography \cite{katz2014private, bogdanov2006average}. Recently, researchers have extended this approach to the study of optimization algorithms, particularly for optimizing quadratic functions \cite{pedregosa2020acceleration, scieur2020universal, paquette2020halting}. The previous works, including \cite{braca2008running} and \cite{berthier2020accelerated}, have considered the case where the gossip algorithm is initialized randomly. Unlike worst-case analysis, which relies on the extremal eigenvalues of the matrix, average-case analysis considers the \textit{expected spectral distribution} of its eigenvalues. This enhances the understanding of the algorithm's convergence speed and facilitates deriving the optimal algorithm.

\subsection{Contributions}

This study focuses on analyzing optimization algorithms for the consensus problem over \textit{regular graphs} by applying average-case analysis perspectives. The main contributions are:

\begin{itemize}
\item \textbf{Derivation of the Optimal Method.} We derive an optimal algorithm for solving the consensus problem on regular graphs under average-case conditions. To do that, we compute orthogonal polynomials w.r.t. expected spectral measure for random regular graphs. As a result, we get the step-sizes for the consensus algorithm. Our optimal algorithm is the equivalent message-passing algorithm, which was introduced in \cite{berthier2020accelerated}. The connection between message-passing algorithm and polynomial-based algorithm was proposed in \cite{berthier2020accelerated}.
\item \textbf{Analysis of the Optimal Method.} A detailed asymptotic analysis is conducted to evaluate the convergence rate of the optimal algorithm. Note that the difference with the message-passing algorithm analysis in \cite{berthier2020accelerated} is the exact convergence rate we obtain using the average-case analysis. This analysis reveals the convergence characteristics of the method, demonstrating its relationship to the Heavy Ball method and allowing us to compare their performances.
\item \textbf{Experiments}. A series of numerical experiments was carried out to evaluate the effectiveness of the optimal method in comparison with other first-order methods.
\end{itemize}

\section{Consensus Problem} \label{sec:prob_form}

Consider a network of agents represented by an undirected finite graph $G = (V, E)$, where $V = \{1, \dots, n\}$ represents the set of vertices (agents) and $E$ represents the set of edges (communication links). Each agent $i$ holds an initial vector $x_0^{(i)} \in \mathbb{R}^d$. We denote by $x_0 = \left(\left(x^{(1)}_0\right)^\top, \dots, \left(x^{(n)}_0\right)^\top\right)^\top$. The goal is to design efficient algorithms that allow each agent to quickly compute the average value $\overline{x}_0 = \frac{1}{n} \sum_{i = 1}^n x_0^{(i)}$, with the constraint that at each iteration of the algorithm, agents can only exchange their vectors with their neighbors.

To achieve consensus on the graph $G$, we solve the following problem starting with the initial vector $x_0$:

\begin{equation} \label{eq:quad_prob}
    \min_{x \in \mathbb{R}^{nd}} f(x) = \frac{1}{2} x^T \mathbf{L} x,
\end{equation}
where $\mathbf{L} = L \otimes I_d$, $\otimes$ denotes the Kronecker product, and $L$ is a \textit{gossip matrix}, which is defined as follows

\begin{definition}
    A gossip matrix $L \in \mathbb{R}^{n \times n}$ on the graph $G = (V, E)$ is a matrix satisfying following properties:
    \begin{enumerate}
        \item $L$ is an $n \times n$ symmetric matrix,
        \item $L$ is positive semi-definite,
        \item $\ker(L ) = \spanop(\mathbf{1})$, where $\mathbf{1} = (1, \dots, 1)^\top$,
        \item $L$ is defined on the edges of the network: $L_{ij} \neq 0$ only if $i = j$ or $(i, j) \in E$.
    \end{enumerate}
\end{definition}

Laplacian is an example of gossip matrix. The typical gossip matrix for a regular graph is $L = I_n - \frac{1}{k} \cdot A$, where $A$ is the adjacency matrix. The detail properties of gossip matrix $L$ of regular graph will be covered in Section \ref{sec:spectrum_reg_graph}.

In this work, we focus on the optimization perspective for the averaging problem on the reformulation in equation \eqref{eq:quad_prob}. This point of view was applied to achieve consensus acceleration in decentralized optimization \cite{metelev2023consensus, gorbunov2022recent, nguyen2023min}.

We consider \textit{first-order methods} or \textit{gradient-based methods} to solve the problem \eqref{eq:quad_prob}. These are methods in which the sequence of iterates $x_t$ is in the span of previous gradients, i.e.,
\begin{equation} \label{eq:first_order_methods}
    x_{t + 1} \in x_0 + \spanop \{ \nabla f(x_0) , \dots, \nabla f(x_t) \}
\end{equation}
or precisely,
\begin{equation} \label{eq:first_order_methods_quad}
    x_{t + 1} \in x_0 + \spanop \{ \mathbf{L} x_0 , \dots, \mathbf{L} x_t \}
\end{equation}

The following lemma explains why the consensus problem can be solved using an optimization approach.
\begin{lemma}
    Suppose that the sequence $\{x_t\}_{t=1}^{\infty}$ is generated by a first-order algorithm of the kind in \eqref{eq:first_order_methods_quad}, starting with $x_0$. If $\lim\limits_{t\to\infty} x_t = x_*$ and $f(x_*) = 0$, then $x_* = \frac{1}{n} \mathbf{1} \mathbf{1}^T x_0$.
\end{lemma}

\begin{proof}
     From \eqref{eq:first_order_methods_quad} it follows that $\forall t \, x_t - x_0 \in \Ima \mathbf{L}$. 
     
      Hence $x_* - x_0 \in \Ima \mathbf{L} = \left\{ x = \left(\left(x^{(1)}\right)^T, \dots, \left(x^{(n)}\right)^T\right)^T \in \mathbb{R}^{nd} \bigg| \sum_{i = 1}^n x^{(i)} = 0 \right\}$.
     
     Then we have $\sum_{i = 1}^n x_*^{(i)} = \sum_{i = 1}^n x_0^{(i)}$.

     Moreover $f(x_*) = 0$ implies that $x_* \in \ker \mathbf{L}$ or $x_*^{(1)} = \cdots = x_*^{(n)}$. Therefore, $\forall i \, x_*^{(i)} = \overline{x}_0 = \frac{1}{n} \sum_{i = 1}^n x_0^{(i)}$.

     This means that the solution of the problem \eqref{eq:quad_prob} using gradient-based methods with the initial vector $x_0$ is the vector $x_* = \frac{1}{n} \mathbf{1} \mathbf{1}^T x_0$.
\end{proof}

\section{Polynomial-Based Iterative Methods} \label{sec:poly_methods}

\subsection{From First-Order Methods to Polynomials}

Our research is based on the classical framework of polynomial-based iterative methods, developed by \cite{fischer2011polynomial} in his work on solving linear systems. This framework reveals an intimate link between first-order methods and polynomials, simplifying the analysis of quadratic objectives. The following lemma illustrates how polynomials can be effectively utilized in the analysis of optimization methods.

\begin{lemma}
    (\cite{hestenes1952methods}) Let $x_t$ be generated by a first-order method. Then there exists a polynomial $P_t$ of degree $t$ such that $P_t(0) = 1$ and it verifies 
    \begin{equation}
        x_t - x_* = P_t(\mathbf{L})(x_0 - x_*)
    \end{equation}
    Following \cite{fischer2011polynomial}, we will refer to this polynomial $P_t$ as the residual polynomial.
\end{lemma}

The proof of the lemma can be found in \cite{pedregosa2021residualpolyopt}. Applying this lemma, we can assign a polynomial to each optimization method and thus determine the method's convergence. Consequently, we can derive accelerated algorithms by selecting a sequence of polynomials that minimizes the distance to the solution. In the following sections, we will discuss several popular optimization methods and their corresponding polynomials.

\subsection{Gradient Descent}

Let $\lambda_{\max} = \lambda_{\max}(L)$ and $\lambda_{\min} = \lambda_{\min}^+(L)$ denote the largest and smallest positive eigenvalues of $L$, respectively. Consider the gradient descent method with optimal step-size as described by \cite{polyak1987introduction}. The iterates are generated as follows:

\begin{equation}  \label{eq:grad_descent}
    x_{t + 1} = x_t - \frac{2}{\lambda_{\max} + \lambda_{\min}} \nabla f(x_t)
\end{equation}

The residual polynomial associated with the gradient descent method has a remarkably simple form. By the properties of $L$, we have $\mathbf{L} x_* = \frac{1}{n} \mathbf{L} \mathbf{1} \mathbf{1}^T x_0 = 0$. We can use this to express the gradient as $\nabla f(x_t) = \mathbf{L} x_t = \mathbf{L} (x_t - x_*)$. Subtracting $x_*$ from both sides of equation \eqref{eq:grad_descent}, we get:

\begin{equation*}
    x_{t + 1} - x_* = \left(I_{nd} - \frac{2}{\lambda_{\max} + \lambda_{\min}} \cdot \mathbf{L}\right) (x_t - x_*) = \cdots = \left(I_{nd} - \frac{2}{\lambda_{\max} + \lambda_{\min}} \cdot \mathbf{L}\right)^{t + 1} (x_0 - x_*)
\end{equation*}

Thus, the residual polynomial of gradient descent method is:

\begin{equation} \label{eq:GD_residual_poly}
    P_t^{\text{GD}}(\lambda) = \left(1 - \frac{2}{\lambda_{\max} + \lambda_{\min}} \lambda\right)^t.
\end{equation}

\subsection{Gradient Descent with Momentum}
Gradient descent with momentum is an optimization method that improves the basic gradient descent algorithm by adding a momentum term to the update rule. This method updates parameters using two components: \textit{momentum term}, which is the difference between the current and previous iterates previous iterate $(x_t - x_{t - 1})$, and the gradient of the objective function $\nabla f(x_t)$.

\begin{algorithm}[H]
\caption{Gradient Descent with Momentum}\label{alg:MomentumGD}
\begin{algorithmic}
\State \textbf{Input:} starting guess $x_0$, step-size $h > 0$ and momentum parameter $m \in (0, 1)$.
\State \textbf{Initialize:} $x_1 = x_0 - \frac{h}{1 + m} \nabla f(x_0)$
\For {$t = 1, 2, \dots$}
\State $x_{t+1} = x_t + m (x_t - x_{t-1}) - h \nabla f(x_t)$
\EndFor
\end{algorithmic}    
\end{algorithm}

\subsubsection*{Polyak Momentum}

Polyak Momentum, also known as the Heavy Ball method, is a widely used optimization technique. Initially developed for solving linear equations and referred to as Frankel's method (\cite{frankel1950convergence, hochstrasser1954anwendung, rutishauser1959theory}), it was later extended to general optimization problems and popularized by Boris Polyak in \cite{polyak1964some, polyak1987introduction}.

The Polyak Momentum method follows the update rules outlined in Algorithm \ref{alg:MomentumGD}, with specific momentum and step-size parameters:

\begin{equation} \label{eq:PM_parameters}
    m = \left(\frac{\sqrt{\lambda_{\max}} - \sqrt{\lambda_{\min}}}{\sqrt{\lambda_{\max}} + \sqrt{\lambda_{\min}}}\right)^2 \quad \text{and} \quad h = \left(\frac{2}{\sqrt{\lambda_{\max}} + \sqrt{\lambda_{\min}}}\right)^2.
\end{equation}

The full algorithm is as follows:

\begin{algorithm}[H]
\caption{Polyak Momentum}\label{alg:PM}
\begin{algorithmic}
\State \textbf{Input:} starting guess $x_0$.
\State \textbf{Initialize:} $x_1 = x_0 - \frac{2}{\lambda_{\max} + \lambda_{\min}} \nabla f(x_0)$
\For {$t = 1, 2, \dots$}
    \State $x_{t+1} = x_t + \left(\frac{\sqrt{\lambda_{\max}} - \sqrt{\lambda_{\min}}}{\sqrt{\lambda_{\max}} + \sqrt{\lambda_{\min}}}\right)^2 (x_t - x_{t-1}) - \left(\frac{2}{\sqrt{\lambda_{\max}} + \sqrt{\lambda_{\min}} }\right)^2 \nabla f(x_t)$
\EndFor
\end{algorithmic}
\end{algorithm}

To derive the residual polynomials associated with the gradient descent method with momentum, we use the Chebyshev polynomials of the first and second kinds. These polynomials are fundamental in various mathematical and applied science domains, including approximation theory, numerical analysis, and spectral methods. In our context, Chebyshev polynomials are employed to express the residual polynomials associated with different optimization methods. Specifically, our optimal method is derived from the Chebyshev polynomials of the second kind, using their orthogonal properties. The Chebyshev polynomials are defined as follows:

\subsubsection*{Chebyshev Polynomials of the First Kind}
The Chebyshev polynomials of the first kind, \( T_n(x) \), are defined by the recurrence relation:
\begin{equation} \label{eq:chebyshev_poly_first_kind}
\begin{aligned}
    T_0(\lambda) &= 1, \\
    T_1(\lambda) &= \lambda, \\
    T_{t+1}(\lambda) &= 2 \lambda T_t(\lambda) - T_{t-1}(\lambda) \quad \text{for} \ t \geq 1.
\end{aligned}
\end{equation}

\subsubsection*{Chebyshev Polynomials of the Second Kind}
Similarly, the Chebyshev polynomials of the second kind, \( U_n(x) \), are defined defined recursively as follows:
\begin{equation} \label{eq:chebyshev_poly_second_kind}
\begin{aligned}
    U_0(\lambda) &= 1, \\
    U_1(\lambda) &= 2 \lambda, \\
    U_{t+1}(\lambda) &= 2 \lambda U_t(\lambda) - U_{t-1}(\lambda) \quad \text{for} \ t \geq 1.
\end{aligned}
\end{equation}

\subsubsection*{Residual polynomial}

The relationship between Chebyshev polynomials and the residual polynomial of gradient descent with momentum is well-established. Specifically, the residual polynomial of gradient descent with momentum can be expressed as a combination of Chebyshev polynomials of the first and second kinds. This relationship is explored in detail in \cite{berthier2020accelerated, pedregosa2021residualmomentum, paquette2023halting}. The following theorem provides the explicit expression for the residual polynomials corresponding to the momentum method.

\begin{theorem}[\cite{berthier2020accelerated, pedregosa2021residualmomentum, paquette2023halting}]
Consider the momentum method \ref{alg:MomentumGD} with step-size $h$ and momentum parameter $m$. The residual polynomial $P_t$ associated with this method can be written in terms of Chebyshev polynomials as
\begin{equation} \label{eq:momentum_poly_via_chebysher}
    P_t^{MGD}(\lambda) = m^{t/2} \left( \frac{2m}{1 + m} T_t(\sigma(\lambda)) + \frac{1 - m}{1 + m} U_t(\sigma(\lambda)) \right),
\end{equation}
with $\sigma(\lambda) = \frac{1}{2 \sqrt{m}} (1 + m - h \lambda)$.
\end{theorem}

Using the known properties of Chebyshev polynomials, this approach helps derive convergence bounds for the momentum method in both worst-case and average-case analyses. A detailed analysis is provided in the appendix.

\section{Framework for average-case analysis} \label{sec:average_framework}

To analyze the average-case consensus problem, we consider a set of graphs and try to estimate the average convergence time of all graphs. Note that for different graphs we obtain different matrices for the optimization problem. Thus, our work is to analyze the convergence rate of the problem \eqref{eq:quad_prob} in the average case with random matrix $L$. We will consider only \textit{oblivious} gradient-based methods of kind \eqref{eq:first_order_methods}, meaning those in which the update coefficients are predetermined and do not rely on prior updates. This excludes some methods like conjugate gradients.

\subsection{Average-case analysis and expected error}
We present a framework for the average-case analysis of random quadratic problems, highlighting how the convergence rate is affected by the matrix spectrum. A practical method for collecting statistical data on the matrix spectrum is by examining its empirical spectral distribution.

\begin{definition}[Empirical/Expected spectral distribution]
Let $L$ be a random matrix with eigenvalues $\{\lambda_1, \dots, \lambda_n\}$. The empirical spectral distribution of $L$, called $\mu_L$, is the probability measure
\begin{equation}
    \mu_L(\lambda) = \frac{1}{n} \sum_{i=1}^n \delta_{\lambda_i} (\lambda),
\end{equation}
where $\delta_{\lambda_i}$ is the Dirac delta, a distribution equal to zero everywhere except at $\lambda_i$ and whose integral over the entire real line is equal to one.

Since $L$ is random, the empirical spectral distribution $\mu_L$ is a random measure. Its expectation over $L$,
\begin{equation}
    \mu = \mathbb{E}_L \left[ \mu_L \right]
\end{equation}
is called the \textit{expected spectral distribution}
\end{definition}

\begin{assumption}
    We assume the $x_0 - x_*$ is independent of $L$ and
    \begin{equation}
        \mathbb{E}(x_0 - x_*)(x_0 - x_*)^{T} = R^2 I.
    \end{equation}
\end{assumption}

\begin{theorem}[\cite{pedregosa2020acceleration}] \label{th:expected_error} Let $x_t$ be generated by a first-order method, associated to the polynomial $P_t$. Then we can decompose the expected error at iteration $t$ as
\begin{equation} \label{eq:expected_error}
    \mathbb{E} \| x_t - x_* \|^ 2 = R^2 \int P_t^2 d\mu.
\end{equation}
\end{theorem}

\subsection{Optimal method}

The framework established in the previous section allows us to investigate the question of optimality concerning average-case complexity. To develop optimal methods, we will first introduce some concepts from the theory of orthogonal polynomials.

\begin{definition}
    Let $\alpha$ be a non-decreasing function such that $\int_{\mathbb{R}} Q d \alpha$ is finite for all polynomials $Q$. We will say that the sequence of polynomials $P_0, P_1, \dots$ is orthogonal with respect to $d\alpha$ if $P_t$ has degree $t$ and
    \begin{equation}
        \int_{\mathbb{R}} P_i P_j d \alpha 
        \begin{cases}
            = 0 ,& \text{if } i \neq j \\
            > 0, & \text{if } i = j.
        \end{cases}
    \end{equation}
    Furthermore, if they verify $P_t(0) = 1$ for all $t$, we call these \textit{residual orthogonal polynomials}.
\end{definition}

\begin{lemma}[Three-term recurrence, \cite{pedregosa2020acceleration}, \cite{fischer2011polynomial}]
Any sequence of residual orthogonal polynomials $P_1$, $P_2$, $\dots$ verifies the following three-term recurrence
\begin{equation} \label{lemma:three_term}
    P_t(\lambda) = (a_t + b_t \lambda) P_{t-1}(\lambda) + (1 - a_t) P_{t-2}(\lambda)
\end{equation}
for some scalars $a_t$, $b_t$ with $a_0 = a_1 = 1$ and $b_0 = 0$.
\end{lemma}

The following theorem is the key result in the average-case analysis of quadratic problems. It explains how to find the optimal method for a given problem spectrum.
\begin{theorem}[Theorem 2.1 from \cite{pedregosa2020acceleration}] \label{th:optimal_method}
    Let $P_t$ be the residual orthogonal polynomials of degree $t$ with respect to the weight function $\lambda d \mu (\lambda)$, and let $a_t$, $b_t$ be the constants associated with
    its three-term recurrence. Then the algorithm
    \begin{equation}
        x_t = x_{t-1} + (1 - a_t) (x_{t-2} - x_{t-1}) + b_t \nabla f(x_{t-1}), 
    \end{equation}
    has the smallest expected error $\mathbb{E} \|x_t - x_*\|^2$ over the class of oblivious first-order methods. Moreover, its expected error is
    \begin{equation}
        \mathbb{E} \|x_t - x_*\|^2 = R^2 \int_{\mathbb{R}} P_t d \mu.
    \end{equation}
\end{theorem}

\section{Optimal method for regular graph} \label{sec:regular_optim}

\subsection{Spectrum of regular graph} \label{sec:spectrum_reg_graph}

\begin{definition}
A regular graph is a graph where all vertices have the same degree, that is, each vertex has the same number of neighbors. A regular graph with vertices of degree $k$ is called a $k$-regular graph.
\end{definition}

\begin{figure}
\begin{subfigure}{.33\textwidth}
  \centering
  \includegraphics[width=\linewidth]{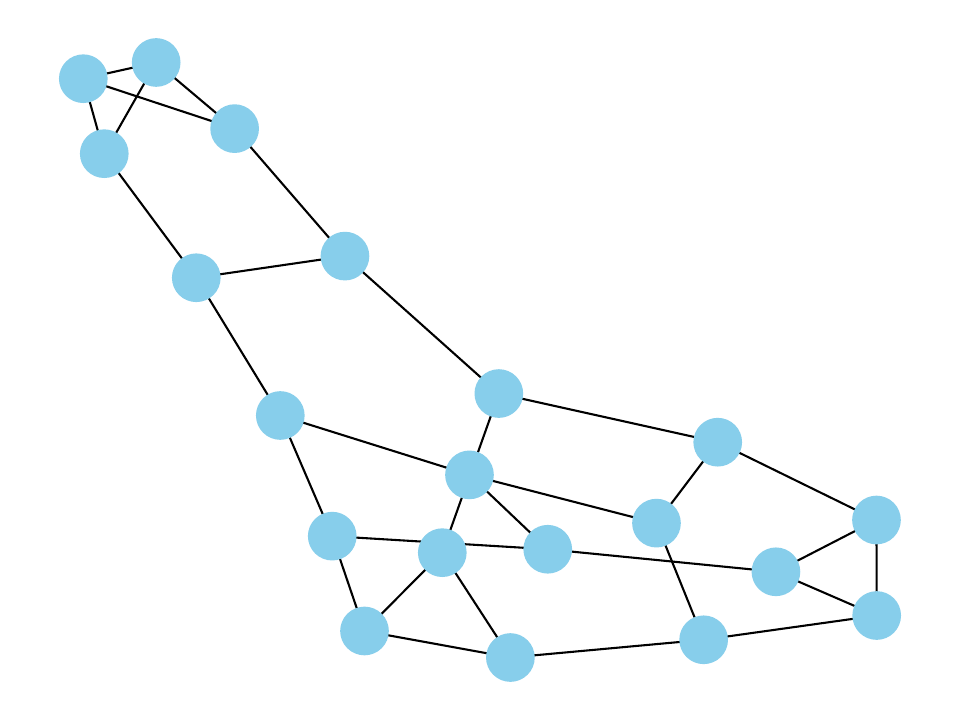}
  \caption{}
  \label{fig:sfig1}
\end{subfigure}%
\begin{subfigure}{.33\textwidth}
  \centering
  \includegraphics[width=\linewidth]{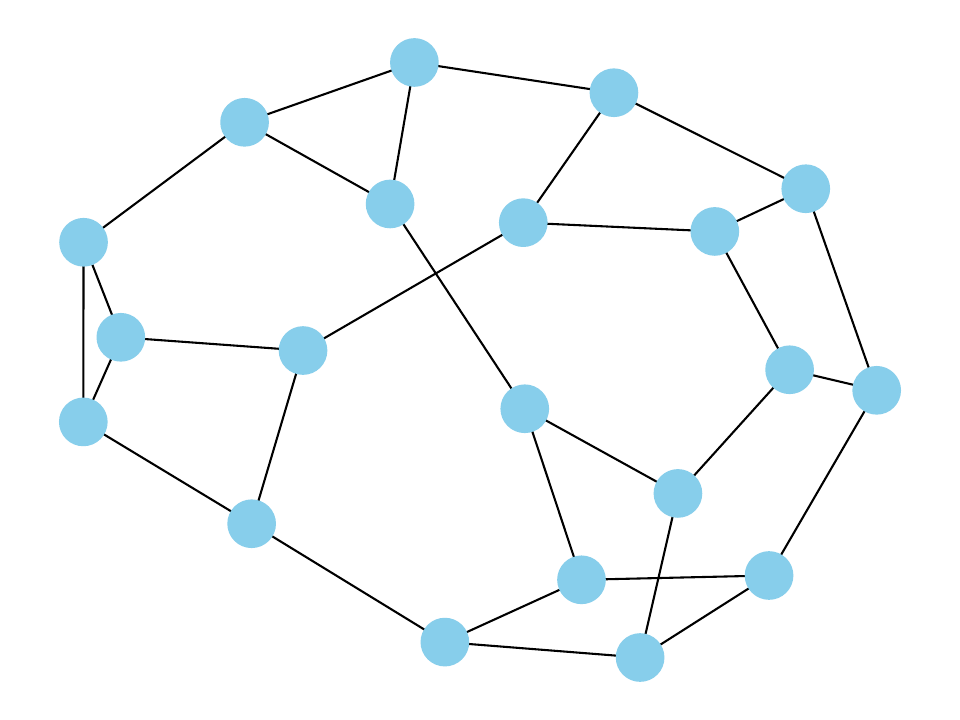}
  \caption{}
  \label{fig:sfig2}
\end{subfigure}
\begin{subfigure}{.33\textwidth}
  \centering
  \includegraphics[width=\linewidth]{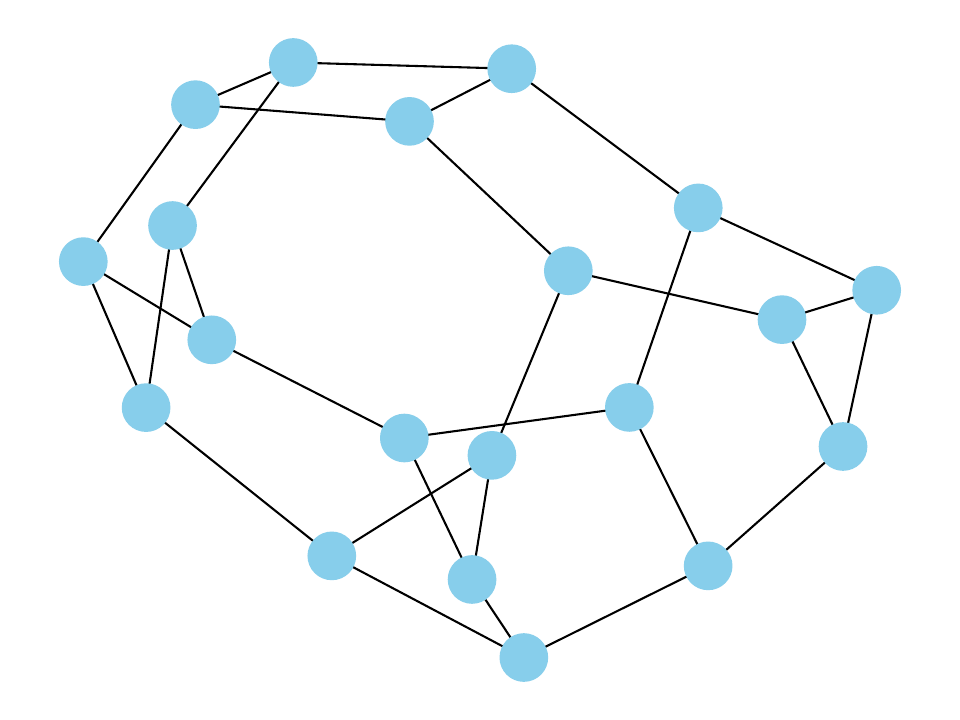}
  \caption{}
  \label{fig:sfig3}
\end{subfigure}
\caption{Regular graphs with $n=20$, $k=3$.}
\label{fig:fig}
\end{figure}

Let $\mathcal{G}_k^{\text{reg}}$ be the set of all $k$-regular graphs ($k \geq 3$) with $n$ vertices.
For each $G \in \mathcal{G}_k^{\text{reg}}$ consider $L(G) = I - \frac{A(G)}{k} $ as a gossip matrix for consensus problem on graph $G$. Then it is known that the expected spectral distribution of $L(G)$ converges to the distribution (from \cite{sole1996spectra, kuhn2015spectra})

\begin{equation} \label{eq:gossip_distr_reg_graph}
    d\mu(\lambda) = \frac{k}{2 \pi} \frac{\sqrt{\frac{4(k - 1)}{k^2} - (1 - \lambda)^2}}{1 - (1 - \lambda)^2} d \lambda,
\end{equation}
which is supported on interval $\left[1-\frac{2\sqrt{k - 1}}{k}, 1+\frac{2\sqrt{k - 1}}{k}\right]$.

The largest and smallest positive eigenvalue of the expected matrix of $L(G)$ are $\lambda_{\max} = 1+\frac{2\sqrt{k - 1}}{k}$ and $\lambda_{\min} = 1-\frac{2\sqrt{k - 1}}{k}$.

\begin{figure}[ht]
    \centering
    \includegraphics[width=4in]{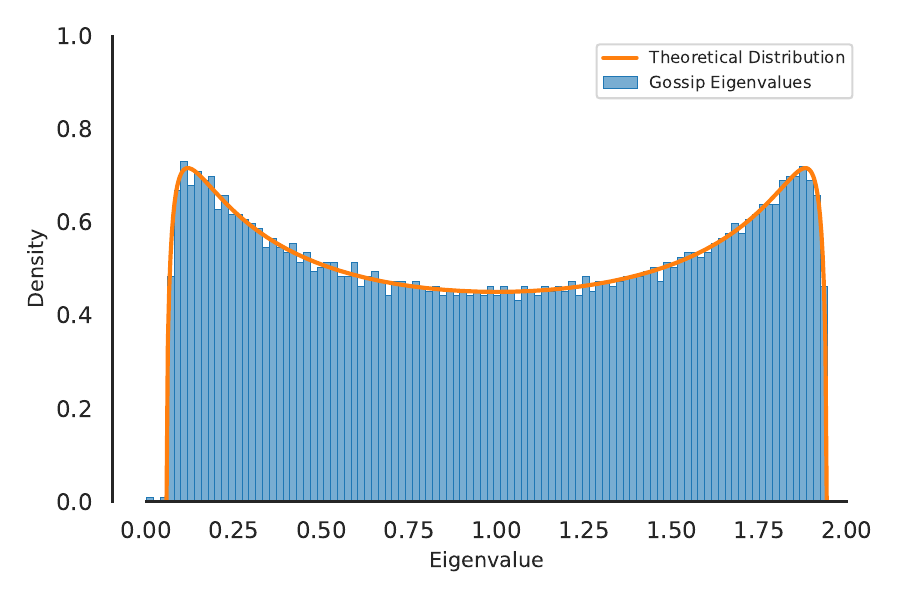}
    \caption{Spectrum of regular graph with $n=5000$, $k=3$.}
    \label{fig:spectrum_plot}
\end{figure}
\subsection{Optimal method}
In this section, we will derive the optimal consensus algorithm for regular graphs. Using the information about spectrum of regular graphs, we construct the sequence of residual orthogonal polynomials with respect to the distribution \eqref{eq:gossip_distr_reg_graph}. The next theorem gives us the recursive representation of these polynomials.

\begin{restatable}[]{theorem}{ResOrthoPolyRG}
\label{th:residual_orthogonal_poly_regular_graph}
    The sequence of residual orthogonal polynomials w.r.t. the weight function $\lambda d\mu(\lambda)$ is
\begin{equation*} 
    \begin{split}
        & Q_0(\lambda) = 1, \\
        & Q_1(\lambda) = 1 - \delta_0 \lambda, \\
        & Q_{t + 1}(\lambda) = \delta_t (1 - \lambda) Q_t(\lambda) + (1 - \delta_t) Q_{t-1}(\lambda),
    \end{split}
\end{equation*}
where $\delta_0 = \frac{k}{k+1}$ and $\delta_t = \left(1 - \frac{k - 1}{k^2} \cdot \delta_{t-1}\right)^{-1}$,  $t \geq 1$.
\end{restatable} 

In appendix \ref{app:orthog_poly}, we also derive the representation of these polynomials via Chebyshev polynomials of the second kind. This kind of representation makes it easier to compute the convergence rate of optimal method.

From the above theorem, we can derive the method with minimum expected error for consensus problem on regular graphs.

\begin{algorithm}[H]
\caption{Optimal average-case method for regular graphs}\label{alg:opt_algo_gossip_reg_graph}
\begin{algorithmic}
\State \textbf{Input:} starting guess $x_0$, regular parameter $k$, $\delta_0 = \frac{k}{k+1}$.
\State \textbf{Initialize:} $x_1 = x_0 - \delta_0 \mathbf{L} x_0$
\For {$t = 1, 2, \dots$}
\State $\delta_t = \left(1 - \frac{k - 1}{k^2} \cdot \delta_{t-1}\right)^{-1}$
\State $x_{t+1} = x_t + (\delta_t - 1) (x_t - x_{t-1}) - \delta_t \mathbf{L} x_{t}$
\EndFor
\end{algorithmic}    
\end{algorithm}

Our optimal method is a momentum-based approach with varying step sizes and momentum, determined by the sequence $\delta_t$. This sequence depends only on the degree parameter of a regular graph. The following lemma demonstrates that this sequence converges.

\begin{lemma} \label{lemma:convergence_delta}
The sequence $\{\delta_t\}_{t=0}^{\infty}$ defined in Theorem \ref{th:residual_orthogonal_poly_regular_graph} is convergent, with
\begin{equation}
    \lim_{t \to \infty} \delta_t = \frac{k}{k - 1}.
\end{equation}
\end{lemma}

From this lemma, we can prove that the coefficients of the optimal average-case Algorithm \ref{alg:opt_algo_gossip_reg_graph} converge to those of Polyak Momentum Method \ref{alg:PM}. Because
$\lim\limits_{t \to \infty} (\delta_t - 1) = \frac{1}{k - 1} = m$ and $\lim\limits_{t \to \infty} \delta_t = \frac{k}{k-1} = h$. This coincides with the more general result for quadratic problems in \cite{scieur2020universal}.
\subsection{Expected errors} 

\subsubsection{Optimal method}
The optimal Algorithm \ref{alg:opt_algo_gossip_reg_graph} aims to efficiently solve consensus problems by leveraging the properties of the gossip matrix associated with random $k$-regular graphs. This method is designed to minimize the expected error over iterations, providing a robust solution for distributed optimization tasks. The following theorem quantifies the expected error for the optimal method when applied to the given problem.

\begin{restatable}[]{theorem}{ExpectedErrorOptMethod}
\label{th:expected_error_optimal_method}
    If we apply Algorithm \ref{alg:opt_algo_gossip_reg_graph} to problem \eqref{eq:quad_prob}, where $L$ is the gossip matrix of random $k$-regular graphs, then
    \begin{equation} \label{eq:expected_error_optimal_method}
    \mathbb{E} \|x_t - x_*\|^2 = \Theta \left( \left(\frac{1}{k - 1}\right)^t \cdot \left( \frac{1}{1 + \frac{2}{k - 2} \left(1 - \frac{1}{(k - 1)^t}\right)}\right)^2 \right).
    \end{equation}
\end{restatable}

\subsubsection{Heavy Ball method}
According to Lemma \ref{lemma:convergence_delta}, our optimal method asymptotically approaches the behavior of the Heavy Ball method. To evaluate this relationship, we derive the expected error of the Heavy Ball method for comparison with the optimal method. The expected error is quantified in the following theorem.

\begin{restatable}[]{theorem}{ExpectedErrorHB}
\label{th:expected_error_HB}
    If we apply Algorithm \ref{alg:PM} to problem \eqref{eq:quad_prob}, where $L$ is the gossip matrix of random $k$-regular graphs, then
    \begin{equation}
    \mathbb{E} \|x_t - x_*\|^2 = \Theta \left( \left( \frac{1}{k - 1} \right)^t\right)
    \end{equation}
\end{restatable}

We observe that both the optimal method and the Heavy Ball method converge at a similar rate. Although the optimal method shows slightly better performance, the difference is insignificant. This conclusion is also observed in the practical experiments in Section \ref{sec:experiments}.

\subsection{Experiments} \label{sec:experiments}
We conduct experiments with the optimal method and the classical methods from Section \ref{sec:poly_methods} and Appendix \ref{app:sup_algo} to evaluate their effectiveness. The distribution of eigenvalues of the gossip matrix changes depending on different values of the parameter $k$ of regular graphs. We present results for several values of $k$ to illustrate the effectiveness of the methods in different cases. The experimental results are shown in Figure \ref{fig:convergence_plots}. For generation of random $k$-regular graphs we use library Networkx~\cite{hagberg2008exploring}.

\begin{figure}[ht]
    \centering
    \includegraphics[width=\linewidth]{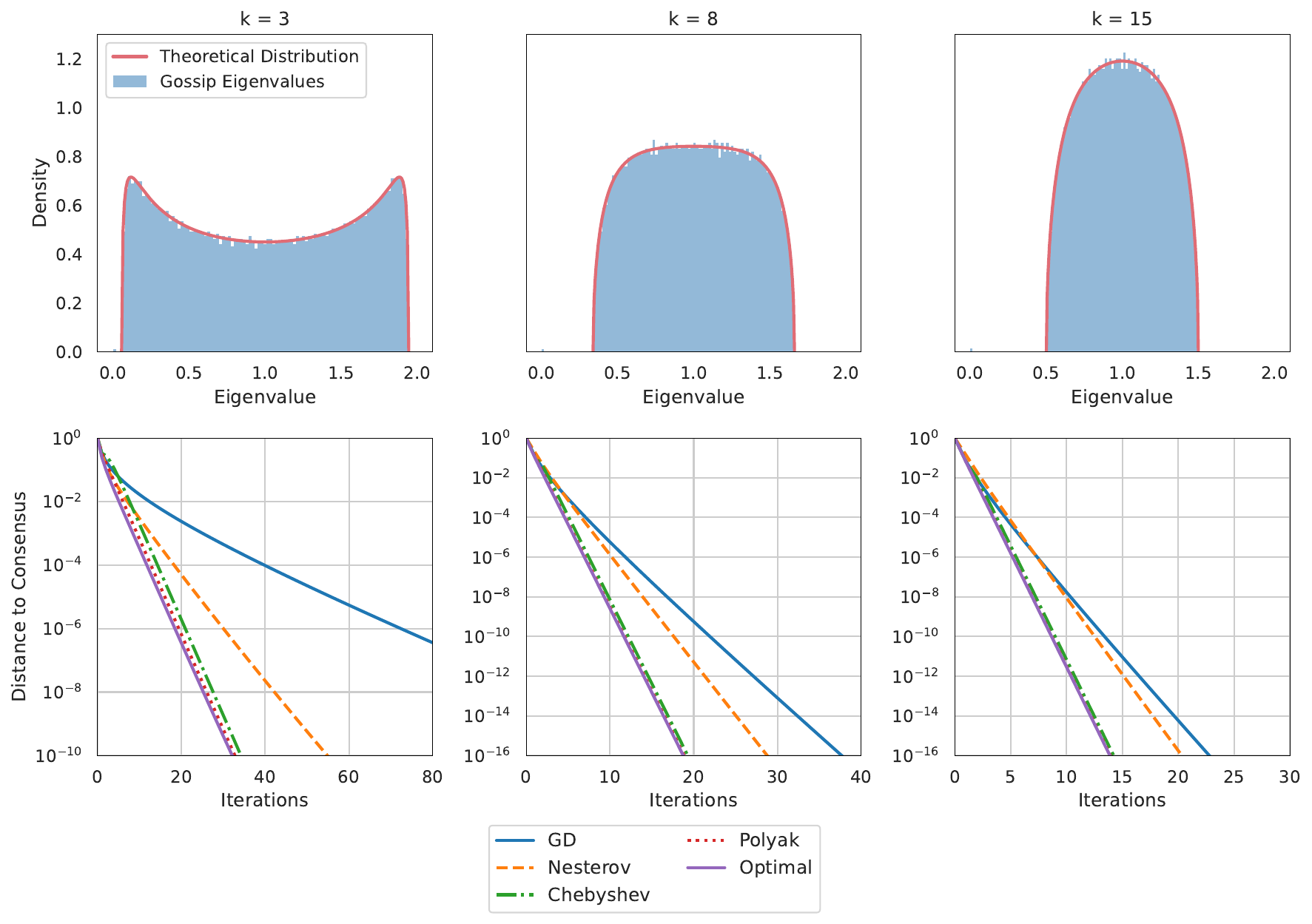}
     \caption{\textbf{Comparison of convergence speeds of algorithms on regular graphs.} The top row shows the spectral distribution of the gossip matrix for $5000$-vertex regular graphs with degree parameters $k = 3$, $8$, and $15$. The bottom row displays the normalized distance to consensus of each algorithm at each iteration, with the y-axis scaled logarithmically. In these experiments, each node vector size is set to $d = 1000$.}
    \label{fig:convergence_plots}
\end{figure}

In the average-case analysis for quadratic optimization, only methods with predefined parameters are typically considered. However, in this work, we also experimented with the conjugate gradient method and compared its performance with our optimal method (see Figure \ref{fig:opt_vs_cg}). From practical experiments, we have drawn the following conclusions:

\begin{itemize}
    \item When the number of vertices in the graph is not very large (e.g., 100-500), the conjugate gradient method exhibits a significantly faster convergence rate compared to our optimal method.
    \item As the number of vertices increases, the convergence speeds of the two methods become nearly identical.
\end{itemize}

This behavior can be explained logically. When the number of vertices is large, the spectrum of the gossip matrix of a random regular graph closely approximates the expected spectral distribution. In this scenario, our optimal method demonstrates its efficiency, as it is designed based on the expected spectral distribution. Consequently, for large-scale problems, our optimal method can achieve the convergence rate of the conjugate gradient method, which is known for its rapid convergence. Overall, our findings highlight the robustness and scalability of our optimal method in handling large-scale consensus problems on regular graphs.

\begin{figure}
    \centering
    \includegraphics[width=\linewidth]{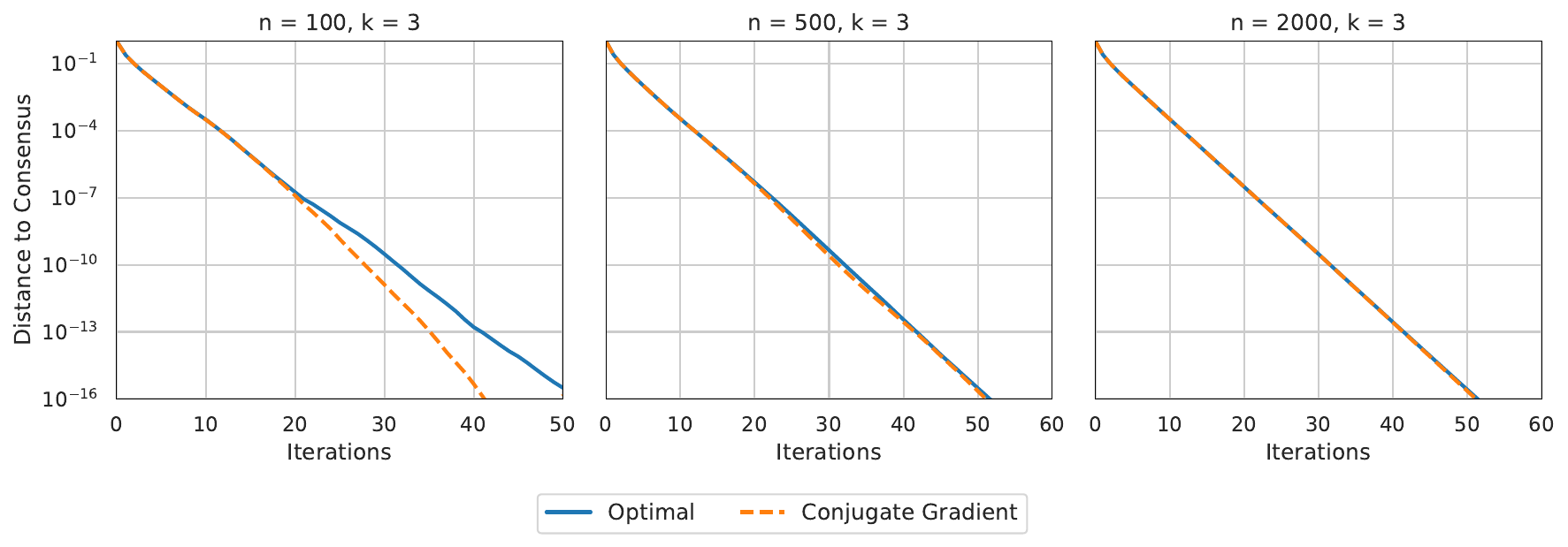}
    \caption{Convergence speeds of the optimal method and the conjugate gradient method}
    \label{fig:opt_vs_cg}
\end{figure}

\section{Conclusions and Future Work}
This work demonstrates that given the spectrum of a regular graph, we can derive orthogonal polynomials corresponding to this spectrum and, consequently, develop an optimal algorithm for the consensus problem in average-case. In addition to this study, we explored the spectrum of gossip matrices in other graph types such as random Erdős-Rényi graphs \cite{ding2010spectral}, scale-free graphs, and small-world graphs \cite{mirchev2017spectra, samukhin2008laplacian}. However, addressing these graphs turned out to be more complex, and we have not yet been able to study optimal algorithms for these types of graphs. Therefore, analyzing the consensus problem on these types of graphs remains an open question for future research.

Furthermore, it has been proven in \cite{scieur2020universal} that if the spectrum of the matrix is strictly positive in the interval $\left(\lambda_{\min},\, \lambda_{\max}\right)$, the optimal method converges to the Polyak method. This raises several intriguing questions: When the gossip matrix spectrum of graph is not continuous and divided into disjoint pieces, such as Laplacian of ring of cliques, is there an algorithm significantly better than Polyak’s method? What will the optimal algorithm be in such cases? Are there connections between the optimal method and Polyak’s method, or with the cyclic Heavy Ball method (\cite{goujaud2022super})? Addressing these questions could provide valuable insights and directions for future research.

\newpage
\subsection*{Acknowledgements}
The research is supported by the Ministry of Science and Higher Education of the Russian Federation (Goszadaniye) No. 075-03-2024-117/6, project No. FSMG-2024-0025.

\bibliographystyle{abbrv}
\bibliography{references}

\begin{thebibliography}{10}

\bibitem{berthier2020accelerated}
R.~Berthier, F.~Bach, and P.~Gaillard.
\newblock Accelerated gossip in networks of given dimension using jacobi polynomial iterations.
\newblock {\em SIAM Journal on Mathematics of Data Science}, 2(1):24--47, 2020.

\bibitem{bogdanov2006average}
A.~Bogdanov, L.~Trevisan, et~al.
\newblock Average-case complexity.
\newblock {\em Foundations and Trends{\textregistered} in Theoretical Computer Science}, 2(1):1--106, 2006.

\bibitem{boyd2006randomized}
S.~Boyd, A.~Ghosh, B.~Prabhakar, and D.~Shah.
\newblock Randomized gossip algorithms.
\newblock {\em IEEE transactions on information theory}, 52(6):2508--2530, 2006.

\bibitem{braca2008running}
P.~Braca, S.~Marano, and V.~Matta.
\newblock Running consensus in wireless sensor networks.
\newblock In {\em 2008 11th International Conference on Information Fusion}, pages 1--6. IEEE, 2008.

\bibitem{ding2010spectral}
X.~Ding and T.~Jiang.
\newblock Spectral distributions of adjacency and laplacian matrices of random graphs.
\newblock {\em The annals of applied probability}, pages 2086--2117, 2010.

\bibitem{d2021acceleration}
A.~d’Aspremont, D.~Scieur, A.~Taylor, et~al.
\newblock Acceleration methods.
\newblock {\em Foundations and Trends{\textregistered} in Optimization}, 5(1-2):1--245, 2021.

\bibitem{fischer2011polynomial}
B.~Fischer.
\newblock {\em Polynomial based iteration methods for symmetric linear systems}.
\newblock SIAM, 2011.

\bibitem{frankel1950convergence}
S.~P. Frankel.
\newblock Convergence rates of iterative treatments of partial differential equations.
\newblock {\em Mathematics of Computation}, 4(30):65--75, 1950.

\bibitem{georgopoulos2014distributed}
L.~Georgopoulos and M.~Hasler.
\newblock Distributed machine learning in networks by consensus.
\newblock {\em Neurocomputing}, 124:2--12, 2014.

\bibitem{gorbunov2022recent}
E.~Gorbunov, A.~Rogozin, A.~Beznosikov, D.~Dvinskikh, and A.~Gasnikov.
\newblock Recent theoretical advances in decentralized distributed convex optimization.
\newblock In {\em High-Dimensional Optimization and Probability: With a View Towards Data Science}, pages 253--325. Springer, 2022.

\bibitem{goujaud2022super}
B.~Goujaud, D.~Scieur, A.~Dieuleveut, A.~B. Taylor, and F.~Pedregosa.
\newblock Super-acceleration with cyclical step-sizes.
\newblock In {\em International Conference on Artificial Intelligence and Statistics}, pages 3028--3065. PMLR, 2022.

\bibitem{hagberg2008exploring}
A.~Hagberg, P.~J. Swart, and D.~A. Schult.
\newblock Exploring network structure, dynamics, and function using networkx.
\newblock Technical report, Los Alamos National Laboratory (LANL), Los Alamos, NM (United States), 2008.

\bibitem{hestenes1952methods}
M.~R. Hestenes, E.~Stiefel, et~al.
\newblock {\em Methods of conjugate gradients for solving linear systems}, volume~49.
\newblock NBS Washington, DC, 1952.

\bibitem{10.1093/comjnl/5.1.10}
C.~A.~R. Hoare.
\newblock {Quicksort}.
\newblock {\em The Computer Journal}, 5(1):10--16, 01 1962.

\bibitem{hochstrasser1954anwendung}
U.~Hochstrasser.
\newblock {\em Die Anwendung der Methode der konjugierten Gradienten und ihrer Modifikationen auf die L{\"o}sung linearer Randwertprobleme}.
\newblock PhD thesis, ETH Zurich, 1954.

\bibitem{katz2014private}
J.~Katz and Y.~Lindell.
\newblock Private-key encryption.
\newblock In {\em Introduction to Modern Cryptography}, pages 77--82. CRC Press, 2014.

\bibitem{knuth1997art}
D.~E. Knuth.
\newblock {\em The art of computer programming}, volume~3.
\newblock Pearson Education, 1997.

\bibitem{kuhn2015spectra}
R.~K{\"u}hn.
\newblock Spectra of random stochastic matrices and relaxation in complex systems.
\newblock {\em Europhysics Letters}, 109(6):60003, 2015.

\bibitem{metelev2023consensus}
D.~Metelev, A.~Rogozin, D.~Kovalev, and A.~Gasnikov.
\newblock Is consensus acceleration possible in decentralized optimization over slowly time-varying networks?
\newblock In {\em International Conference on Machine Learning}, pages 24532--24554. PMLR, 2023.

\bibitem{mirchev2017spectra}
M.~Mirchev.
\newblock On the spectra of scale-free and small-world networks, 2017.

\bibitem{nemirovski1995information}
A.~Nemirovski.
\newblock Information-based complexity of convex programming.
\newblock {\em Lecture notes}, 834, 1995.

\bibitem{nesterov2004}
Y.~Nesterov.
\newblock {\em Introductory Lectures on Convex Optimization}.
\newblock Springer, 2004.

\bibitem{nguyen2023min}
N.~T. Nguyen, A.~Rogozin, D.~Metelev, and A.~Gasnikov.
\newblock Min-max optimization over slowly time-varying graphs.
\newblock In {\em Doklady Mathematics}, volume 108, pages S300--S309. Springer, 2023.

\bibitem{paquette2020halting}
C.~Paquette, B.~van Merri{\"e}nboer, E.~Paquette, and F.~Pedregosa.
\newblock Halting time is predictable for large models: A universality property and average-case analysis.
\newblock {\em arXiv preprint arXiv:2006.04299}, 2020.

\bibitem{paquette2023halting}
C.~Paquette, B.~van Merri{\"e}nboer, E.~Paquette, and F.~Pedregosa.
\newblock Halting time is predictable for large models: A universality property and average-case analysis.
\newblock {\em Foundations of Computational Mathematics}, 23(2):597--673, 2023.

\bibitem{pedregosa2021residualmomentum}
F.~Pedregosa.
\newblock Momentum: when chebyshev meets chebyshev.
\newblock \url{http://fa.bianp.net/blog/2020/momentum/}, 2020.

\bibitem{pedregosa2021residualpolyopt}
F.~Pedregosa.
\newblock Residual polynomials and the chebyshev method.
\newblock \url{http://fa.bianp.net/blog/2020/polyopt/}, 2020.

\bibitem{pedregosa2020acceleration}
F.~Pedregosa and D.~Scieur.
\newblock Acceleration through spectral density estimation.
\newblock In {\em International Conference on Machine Learning}, pages 7553--7562. PMLR, 2020.

\bibitem{polyak1987introduction}
B.~Polyak.
\newblock {\em Introduction to Optimization}.
\newblock Optimization Software, Inc. Publications Division, New York, July 2020.

\bibitem{polyak1964some}
B.~T. Polyak.
\newblock Some methods of speeding up the convergence of iteration methods.
\newblock {\em Ussr computational mathematics and mathematical physics}, 4(5):1--17, 1964.

\bibitem{rogozin2023decentralized}
A.~Rogozin, A.~Gasnikov, A.~Beznosikov, and D.~Kovalev.
\newblock Decentralized convex optimization over time-varying graphs.
\newblock In {\em Encyclopedia of Optimization}, pages 1--17. Springer, 2023.

\bibitem{rutishauser1959theory}
H.~Rutishauser.
\newblock Theory of gradient methods.
\newblock {\em Refined iterative methods for computation of the solution and the eigenvalues of self-adjoint boundary value problems}, pages 24--49, 1959.

\bibitem{samukhin2008laplacian}
A.~Samukhin, S.~Dorogovtsev, and J.~Mendes.
\newblock Laplacian spectra of, and random walks on, complex networks: Are scale-free architectures really important?
\newblock {\em Physical Review E}, 77(3):036115, 2008.

\bibitem{scaman2017optimal}
K.~Scaman, F.~Bach, S.~Bubeck, Y.~T. Lee, and L.~Massouli{\'e}.
\newblock Optimal algorithms for smooth and strongly convex distributed optimization in networks.
\newblock In {\em international conference on machine learning}, pages 3027--3036. PMLR, 2017.

\bibitem{scieur2020universal}
D.~Scieur and F.~Pedregosa.
\newblock Universal average-case optimality of polyak momentum.
\newblock In {\em International conference on machine learning}, pages 8565--8572. PMLR, 2020.

\bibitem{smale1983average}
S.~Smale.
\newblock On the average number of steps of the simplex method of linear programming.
\newblock {\em Mathematical programming}, 27(3):241--262, 1983.

\bibitem{sole1996spectra}
P.~Sol{\'e} et~al.
\newblock Spectra of regular graphs and hypergraphs and orthogonal polynomials.
\newblock {\em European Journal of Combinatorics}, 17(5):461--477, 1996.

\bibitem{spielman2004smoothed}
D.~A. Spielman and S.-H. Teng.
\newblock Smoothed analysis of algorithms: Why the simplex algorithm usually takes polynomial time.
\newblock {\em Journal of the ACM (JACM)}, 51(3):385--463, 2004.

\bibitem{tsianos2012consensus}
K.~I. Tsianos, S.~Lawlor, and M.~G. Rabbat.
\newblock Consensus-based distributed optimization: Practical issues and applications in large-scale machine learning.
\newblock In {\em 2012 50th annual allerton conference on communication, control, and computing (allerton)}, pages 1543--1550. IEEE, 2012.

\bibitem{vershik1983estimation}
A.~M. Vershik and P.~Sporyshev.
\newblock Estimation of the mean number of steps in the simplex method, and problems of asymptotic integral geometry.
\newblock In {\em Dokl. Akad. Nauk SSSR}, volume 271, pages 1044--1048, 1983.

\end{thebibliography}
\newpage
\appendix{\Large Supplementary Material}

\section{Derivation of optimal method} \label{app:orthog_poly}
Let us consider a measure
\begin{equation} \label{eq:modified_semicircle_measure}
    d\nu(\lambda) = \frac{2}{\pi} \frac{\sqrt{1 - \lambda^2}}{\rho(\lambda)} d\lambda.
\end{equation}
The polynomials orthogonal w.r.t. measure \eqref{eq:modified_semicircle_measure} are called Bernstein – Szegő polynomials.

\begin{lemma}[Corollary 1 from \cite{sole1996spectra}] \label{lemma:general_orthog_poly}
Let
\begin{equation} \label{eq:rho_cos}
    \rho(\cos{\theta}) = |h(e^{i\theta})|^2
\end{equation}
for some complex polynomial $h$ of the same degree as $\rho$. If $h(z) = a + bz + cz^2$ satisfies \eqref{eq:rho_cos} and has real coefficients $a$, $b$, $c$, then $P_0(\lambda) = 1$, $P_1(\lambda) = a U_1(\lambda) + bU_0(\lambda)$, and for $t \geq 2$,
\begin{equation} \label{eq:general_form_for_orthogonal_poly}
    P_t(\lambda) = a U_t(\lambda) + b U_{t-1}(\lambda) + cU_{t-2}(\lambda)
\end{equation}
constitute a family of orthogonal polynomials with respect to $\nu(\lambda)$.
\end{lemma}

To simplify the presentation of expressions, we use the notation $q = k - 1$. 
\begin{theorem} \label{th:orthogonal_poly_reg_graph_cheb_form}
    The sequence of polynomials
\begin{equation} \label{eq:orthogonal_poly_reg_graph_cheb_form}
\begin{split}
    & P_0(\lambda) = \sqrt{q}, \\
    & P_t(\lambda)  = \sqrt{q} \cdot U_t\left(\frac{k (1 - \lambda)}{2 \sqrt{q}} \right) + U_{t-1}\left(\frac{k (1 - \lambda)}{2 \sqrt{q}} \right), \; t \geq 1 
\end{split}
\end{equation}
is orthogonal w.r.t. the weight function $\lambda \,d\mu(\lambda)$, where $\mu(\lambda)$ is defined in \eqref{eq:gossip_distr_reg_graph}.
\end{theorem}

\begin{proof}
From \eqref{eq:gossip_distr_reg_graph}, we know that the measure $\mu$ has the following density:
\begin{equation} 
    d\mu(\lambda) = \frac{k}{2 \pi} \frac{\sqrt{\frac{4q}{k^2} - (1 - \lambda)^2}}{1 - (1 - \lambda)^2} \,d \lambda,
\end{equation}
This density function is symmetric at 1, let's transform it into a form that is symmetric at 0
\begin{equation}
    d\mu(1 - \lambda) = \frac{k}{2 \pi} \cdot \frac{\sqrt{\frac{4q}{k^2} - \lambda^2}}{1 - \lambda^2} \,d \lambda,
\end{equation}
Now, let's normalize the measure on the interval $\left[-1, 1\right]$
\begin{equation}
    d\mu\left(1 - \frac{2 \sqrt{q} }{k} \cdot \lambda\right) = \frac{k}{2 \pi} \frac{\sqrt{\frac{4q}{k^2} - \frac{4q}{k^2} \lambda^2}}{1 - \frac{4q}{k^2} \lambda^2} \,d \lambda = \frac{\sqrt{q}}{\pi} \cdot \frac{\sqrt{1 - \lambda^2}}{1-\frac{4q}{k^2} \lambda^2} d \lambda,
\end{equation}
Let $\tau(\lambda) = 1 - \frac{2 \sqrt{q}}{k} \cdot \lambda $, then

\begin{equation}
    \tau(\lambda) \,d \mu \left(\tau(\lambda)\right) = \frac{\sqrt{q}}{\pi} \cdot \frac{\sqrt{1 - \lambda^2}}{1+\frac{2\sqrt{q}}{k} \lambda} d \lambda,
\end{equation}
Up to a constant multiple, the measure $\tau(\lambda) \,d \mu \left(\tau(\lambda)\right)$ has the form $\frac{\sqrt{1 - \lambda^2}}{k + 2\sqrt{q} \lambda}$. Thus, $\rho(\lambda) = k + 2\sqrt{q} \lambda$. A straightforward calculation yields $h(z) = \sqrt{q} + z$. By Lemma \ref{lemma:general_orthog_poly} the polynomials
\begin{equation}
    A_t(\lambda) = \sqrt{q} U_t(\lambda) + U_{t-1}(\lambda), \text{ for } t \geq 1 \text{ and }A_0(\lambda) = \sqrt{q}
\end{equation}

form a family orthogonal with respect to $\tau(\lambda) \,d \mu \left(\tau(\lambda)\right)$, i.e.,
\begin{equation}
        \int_{-1}^{1} A_i(\lambda) \cdot A_j(\lambda) \cdot \tau(\lambda) \,d \mu \left(\tau(\lambda)\right) = 0, \text{ if } i \neq j.
\end{equation}
Note that here we take $A_0(\lambda) = \sqrt{q}$ instead of $A_0(\lambda) = 1$. This does not affect orthogonality, but such choice of $A_0(\lambda)$ is more convenient for further derivations.

Therefore,
\begin{equation}
        \int_{1 -\frac{2 \sqrt{q}}{k}}^{1 + \frac{2 \sqrt{q}}{k}} A_i \left(\frac{k(1 - \lambda)}{2 \sqrt{q}} \right) A_j \left(\frac{k(1 - \lambda)}{2 \sqrt{q}} \right) \lambda d \mu (\lambda) = 0, \text{ if } i \neq j.
\end{equation}
Let $P_t(\lambda) = A_t \left(\frac{k(1 - \lambda)}{2 \sqrt{q}} \right)$, then
\begin{equation}
    \int_{1-\frac{2 \sqrt{q}}{k}}^{1+\frac{2 \sqrt{q}}{k}} P_i(\lambda) P_j(\lambda) \lambda d\mu (\lambda) = 0, \text{ if } i \neq j.
\end{equation}
Therefore, the polynomials
\begin{equation}
    P_t(\lambda) = A_t \left(\frac{k(1 - \lambda)}{2 \sqrt{q}} \right) = \sqrt{q} \cdot U_t\left(\frac{k(1 - \lambda)}{2 \sqrt{q}} \right) + U_{t-1}\left(\frac{k(1 - \lambda)}{2 \sqrt{q}} \right), \text{ for } t \geq 1 \text{ and } P_0(\lambda) = \sqrt{q}
\end{equation}
form a family orthogonal with respect to $\lambda d \mu(\lambda))$.

\end{proof}

\begin{theorem} \label{th:recur_orthogonal_poly}
    The sequence of polynomials
\begin{equation} \label{eq:recur_orthogonal_poly_reg_graph}
    \begin{split}
        & P_0(\lambda) = \sqrt{q}, \\
        & P_1(\lambda) = k (1 - \lambda) + 1, \\
        & P_{t + 1}(\lambda) = \frac{k (1 - \lambda)}{\sqrt{q}} \cdot P_t(\lambda) - P_{t-1}(\lambda)
    \end{split}
\end{equation}
is orthogonal w.r.t. the weight function $\lambda \,d\mu(\lambda)$
\end{theorem}

\begin{proof}
    From theorem \ref{th:orthogonal_poly_reg_graph_cheb_form} we know that the sequence of polynomials 
\begin{equation*}
\begin{split}
    & P_0(\lambda) = \sqrt{q}, \\
    & P_t(\lambda)  = \sqrt{q} \cdot U_t\left(\frac{k(1-\lambda)}{2   \sqrt{q}} \right) + U_{t-1}\left(\frac{k(1-\lambda)}{2   \sqrt{q}} \right), \; t \geq 1 
\end{split}
\end{equation*}
is orthogonal w.r.t. the weight function $\lambda \,d\mu(\lambda)$.
It is straightforward to verify that the recursion holds for $t=1$. Now, consider $t \geq 2$. Then, 
\begin{equation} \label{eq:opt_poly_via_chebyshev}
    P_{t+1}(\lambda)  = \sqrt{q} \cdot U_{t+1}\left(\frac{k(1-\lambda)}{2   \sqrt{q}} \right) + U_{t}\left(\frac{k(1-\lambda)}{2   \sqrt{q}} \right).
\end{equation}
From \eqref{eq:chebyshev_poly_second_kind}:
\begin{equation*}
        U_{t+1}(\lambda) = 2 \lambda U_t(\lambda) - U_{t-1}(\lambda) 
\end{equation*}
Let $\sigma(\lambda) = \frac{k(1-\lambda)}{2 \sqrt{q}}$. We have
\begin{equation*}
\begin{split}
    U_{t+1}\left(\sigma(\lambda) \right) &= 2 \sigma(\lambda) \cdot U_t\left(\sigma(\lambda)\right) - U_{t-1}\left(\sigma(\lambda) \right) \\
    \sqrt{q} \cdot U_{t+1}\left(\sigma(\lambda) \right) &= k(1-\lambda) \cdot U_t\left(\sigma(\lambda) \right) - \sqrt{q} \cdot U_{t-1}\left(\sigma(\lambda) \right) \\
     &= \frac{k(1-\lambda)}{\sqrt{q}} \cdot \left( \sqrt{q} \cdot U_t\left(\sigma(\lambda) \right) + U_{t-1}\left(\sigma(\lambda) \right) \right) - \frac{k(1-\lambda)}{\sqrt{q}} \cdot U_{t-1}\left(\sigma(\lambda) \right) - \sqrt{q} \cdot U_{t-1}\left(\sigma(\lambda) \right) \\
     &=\frac{k(1-\lambda)}{\sqrt{q}} \cdot P_t(\lambda) - \frac{k(1-\lambda)}{\sqrt{q}} \cdot U_{t-1}\left(\sigma(\lambda)\right)   - \sqrt{q} \cdot U_{t-1}\left(\sigma(\lambda)\right)
\end{split} 
\end{equation*}
Therefore, applying to \eqref{eq:opt_poly_via_chebyshev} we get
\begin{equation*}
\begin{split}
     P_{t+1}(\lambda) &= \frac{k(1-\lambda)}{\sqrt{q}} \cdot P_t(\lambda) - \frac{k(1-\lambda)}{\sqrt{q}} \cdot U_{t-1}\left(\sigma(\lambda) \right) - \sqrt{q} \cdot U_{t-1}\left(\sigma(\lambda) \right) + U_{t}\left(\sigma(\lambda) \right) \\
     &= \frac{k(1-\lambda)}{\sqrt{q}} \cdot P_t(\lambda) - \sqrt{q} \cdot U_{t-1}\left(\sigma(\lambda) \right) - U_{t-2}\left(\sigma(\lambda) \right) \\ 
     &= \frac{k(1-\lambda)}{\sqrt{q}} \cdot P_t(\lambda) - P_{t-1}(\lambda)
\end{split} 
\end{equation*}
\end{proof}

\ResOrthoPolyRG*

\begin{proof}
Let $P_t(\lambda)$ be the sequence of polynomials defined by
\begin{equation} \label{eq:rec_poly}
    \begin{split}
        & P_0(\lambda) = \sqrt{q}, \\
        & P_1(\lambda) = k (1 - \lambda) + 1, \\
        & P_{t + 1}(\lambda) = \frac{k (1 - \lambda)}{\sqrt{q}} P_t(\lambda) - P_{t-1}(\lambda).
    \end{split}
\end{equation}
By Theorem \ref{th:recur_orthogonal_poly}, the sequence of polynomials $P_t$ is orthogonal with respect to the weight function $\lambda \, d\mu(\lambda)$. \\
Consider $Q_t(\lambda) = \frac{P_t(\lambda)}{P_t(0)}$. It follows that $Q_t(0) = 1$ for all $t$. Moreover, the sequence $Q_t$ is also orthogonal with respect to $\lambda \, d\mu(\lambda)$. \\
From the recurrence relation, we have:
\begin{align*}
    P_{t + 1}(0) \cdot Q_{t + 1}(\lambda) &= \frac{k (1 - \lambda)}{\sqrt{q}} \cdot P_t(0) \cdot Q_t(\lambda) - P_{t - 1}(0) \cdot Q_{t-1}(\lambda), \\
    Q_{t+1} (\lambda) &= \frac{k (1 - \lambda)}{\sqrt{q}} \cdot \frac{P_t(0)}{P_{t+1}(0)} \cdot Q_t(\lambda) - \frac{P_{t-1}(0)}{P_{t+1}(0)} \cdot Q_{t-1}(\lambda), \\
    Q_{t+1} (\lambda) &= \frac{k (1 - \lambda)}{\sqrt{q}} \cdot \frac{P_t(0)}{P_{t+1}(0)} \cdot Q_t(\lambda) + \left(1 - \frac{k }{\sqrt{q}} \cdot \frac{P_t(0)}{P_{t+1}(0)} \right) \cdot Q_{t-1}(\lambda).
\end{align*}
Let $\delta_t = \frac{k}{\sqrt{q}} \cdot \frac{P_t(0)}{P_{t+1}(0)}$. Then,
\begin{equation} \label{eq:orthogonal_recur_eq}
    Q_{t + 1}(\lambda) = \delta_t (1 - \lambda) Q_t(\lambda) + (1 - \delta_t) Q_{t-1}(\lambda).
\end{equation}
Next, we determine $\delta_t$:
\begin{align*}
    - \frac{P_{t-1}(0)}{P_{t+1}(0)} &= 1 - \frac{k}{\sqrt{q}} \cdot \frac{P_t(0)}{P_{t+1}(0)} \text{ (substituting $\lambda = 0$ into \eqref{eq:rec_poly}}), \\
    - \frac{P_{t-1}(0)}{P_t(0)} \cdot \frac{P_t(0)}{P_{t+1}(0)} &= 1 - \frac{k }{\sqrt{q}} \cdot \frac{P_t(0)}{P_{t+1}(0)}, \\
    - \frac{q}{k^2} \cdot \delta_{t-1} \cdot \delta_{t} &= 1 - \delta_t, \\
     \delta_t \left(1 - \frac{q}{k^2} \cdot \delta_{t-1} \right) &= 1.
\end{align*}
Hence,
\begin{equation} \label{eq:delta_eq}
    \delta_t = \left(1 - \frac{q}{k^2} \cdot \delta_{t-1}\right)^{-1}, \quad \delta_0 = \frac{k}{\sqrt{q}} \cdot \frac{P_0 (0)}{P_1 (0)} = \frac{k}{\sqrt{q}} \cdot \frac{\sqrt{q}}{k + 1} = \frac{k}{k+1}.
\end{equation}
The theorem is established by utilizing \eqref{eq:orthogonal_recur_eq} and \eqref{eq:delta_eq}.
\end{proof}

\section{Convergence rates}

\subsection{Convergence rate of optimal method}

\ExpectedErrorOptMethod*

\begin{proof}
    
The residual polynomial for the optimal method is as follows:
\begin{align*}
    & Q_0(\lambda) = 1, \\
    & Q_1(\lambda) = 1 - \delta_0 \lambda, \\
    & Q_{t + 1}(\lambda) = \delta_t (1 - \lambda) Q_t(\lambda) + (1 - \delta_t) Q_{t-1}(\lambda),
\end{align*} 
where $\delta_0 = \frac{k}{k+1}$ and $\delta_t = \left(1 - \frac{k - 1}{k^2} \cdot \delta_{t-1}\right)^{-1}$,  $t \geq 1$. \\
From the proof of Theorem \ref{th:residual_orthogonal_poly_regular_graph}
\begin{equation*}
    Q_t(\lambda) = \frac{P_t(\lambda)}{P_t(0)} \text{ and } \delta_t = \frac{k}{\sqrt{q}} \cdot \frac{P_t(0)}{P_{t+1}(0)}, 
\end{equation*}
where $P_t(\lambda)$ is defined in Theorem \ref{th:orthogonal_poly_reg_graph_cheb_form}. Thus,
\begin{align*}
    Q_t(\lambda) &= \frac{\sqrt{q} U_t\left(\frac{k(1-\lambda)}{2 \sqrt{q}} \right) + U_{t-     1}\left(\frac{k(1-\lambda)}{2 \sqrt{q}} \right)}{P_t(0)} \\
                 &= \frac{1}{P_0(0)} \cdot \frac{P_0(0)}{P_1(0)} \cdots \frac{P_{t-1}(0)}{P_t(0)} \left( \sqrt{q} U_t\left(\frac{k(1-\lambda)}{2 \sqrt{q}} \right) + U_{t-1}\left(\frac{k(1-\lambda)}{2 \sqrt{q}} \right) \right) \\
                 &= \frac{\delta_0 \cdots \delta_{t-1}}{\sqrt{q}} \left( \frac{\sqrt{q}}{k} \right)^t  \left( \sqrt{q} U_t\left(\frac{k(1-\lambda)}{2 \sqrt{q}} \right) + U_{t-1}\left(\frac{k(1-\lambda)}{2 \sqrt{q}} \right) \right)
\end{align*}
Hence, 
\begin{equation*}
   Q_t(\lambda) = \frac{\delta_0 \cdots \delta_{t-1}}{\sqrt{q}} \left( \frac{\sqrt{q}}{k} \right)^t  \left( \sqrt{q} U_t\left( \sigma\left(\lambda\right)\right) + U_{t-1}\left( \sigma\left(\lambda\right) \right) \right),
\end{equation*} where 
\begin{equation*}
    \sigma(\lambda) = \frac{k(1 - \lambda)}{2 \sqrt{q}}.
\end{equation*}
Let
\begin{equation*}
    \widetilde{Q}_t(\lambda) = \frac{\delta_0 \cdots \delta_{t-1}}{\sqrt{q}} \left( \frac{\sqrt{q}}{k} \right)^t  \left( \sqrt{q} U_t(\lambda) + U_{t-1}(\lambda) \right),
\end{equation*}
Then,
\begin{equation*}
    Q_t(\lambda) = \widetilde{Q}_t\left(\sigma(\lambda)\right).
\end{equation*}
Applying Theorem \ref{th:expected_error}, we have
\begin{align*}
    \mathbb{E} \|x_t - x_*\|^2 &= R^2 \int_{\lambda_{\min}}^{\lambda_{\max}} Q_t^2(\lambda) d\mu(\lambda) \\
    &= R^2 \cdot \frac{k}{2\pi} \int_{\lambda_{\min}}^{\lambda_{\max}} Q_t^2(\lambda) \frac{\sqrt{4 \cdot \frac{k-1}{k^2} -(1-\lambda)^2}}{1 - (1-\lambda)^2} \,d\lambda\\
    &= R^2 \cdot \frac{k}{2\pi} \int_{\lambda_{\min}}^{\lambda_{\max}} \widetilde{Q}_t^2\left(\sigma(\lambda)\right) \frac{\sqrt{4 \cdot \frac{k-1}{k^2} -(1-\lambda)^2}}{1 - (1-\lambda)^2} \,d\lambda\\
\end{align*}
By substitution $u = \sigma(\lambda)$, we get:
\begin{align*}
    \mathbb{E} \|x_t - x_*\|^2 &= R^2 \cdot \frac{2q}{k \pi} \int_{-1}^{1} \widetilde{Q}_t^2(u) \cdot \frac{\sqrt{1 - u^2}}{1 - \frac{4q}{k^2} \cdot u^2} du, \quad \text{where } M = \frac{4q}{k^2} \\
    &= R^2 \cdot \frac{2 \cdot \delta_0^2 \cdots \delta_{t-1}^2}{k\pi} \left( \frac{\sqrt{q}}{k} \right)^{2t}  \int_{-1}^{1} \left( q U_t^2(u) + 2 \sqrt{q} U_t(u) U_{t-1}(u) + U_{t-1}^2(u) \right) \frac{\sqrt{1 - u^2}}{1 - M u^2} du \\
\end{align*}
Using the substitution $u = \cos{\theta}$ and the properties of Chebyshev polynomials, we can express the expected error of optimal method as follows
\begin{align}
    \mathbb{E} \|x_t - x_*\|^2 &= R^2 \cdot C_t \cdot I,
\end{align}
where
\begin{align*}
    C_t = \frac{2 \cdot \delta_0^2 \cdots \delta_{t-1}^2}{k\pi} \left( \frac{\sqrt{q}}{k} \right)^{2t},
\end{align*}
and
\begin{align*}
    I = \int_{0}^{\pi} \left( q  \frac{\sin^2\left((t+1)\theta\right)}{\sin^2{\theta}} + 2 \sqrt{q} \frac{\sin\left((t+1)\theta\right) \sin(t \theta)}{\sin^2{\theta}} + \frac{\sin^2 \left(t\theta\right)}{\sin^2{\theta}} \right) \frac{\sin^2{\theta}}{1 - M \cos^2{\theta}} d \theta
\end{align*}
Since $1 \leq \frac{1}{1 - M \cos^2(\theta)} \leq \frac{1}{1-M}$, we have:
\begin{align*}
    J \leq I \leq \frac{J}{1 - M}
\end{align*}
where
\begin{align*}
    J &= \int_{0}^{\pi} \left( q  \frac{\sin^2\left((t+1)\theta\right)}{\sin^2{\theta}} + 2 \sqrt{q} \frac{\sin\left((t+1)\theta\right) \sin(t \theta)}{\sin^2{\theta}} + \frac{\sin^2 \left(t\theta\right)}{\sin^2{\theta}} \right) \sin^2{\theta} \,d \theta \\
    &=\int_{0}^{\pi} \left( q  \sin^2\left((t+1)\theta\right) + 2 \sqrt{q} \sin\left((t+1)\theta\right) \sin(t \theta) + \sin^2 \left(t\theta\right) \right) d \theta.
\end{align*}
We know that:
\begin{align*}
    &\int_0^{\pi} \sin^2((t + 1) \theta) \,d\theta = \frac{\pi}{2} \\
    &\int_0^{\pi} \sin((t+ 1)\theta) \sin(t \theta) \,d\theta = 0, \\
    &\int_0^{\pi} \sin^2(t \theta) \,d\theta = \frac{\pi}{2}.
\end{align*}
Therefore, $J = \frac{k\pi}{2}$ and
\begin{align*}
    R^2 \cdot C_t \cdot \frac{k \pi}{2} &\leq \mathbb{E} \|x_t - x_* \|^2 \leq R^2 \cdot C_t \cdot \frac{k \pi}{2 (1 - M)} \\
    R^2 \cdot \prod_{i = 0}^{t-1} \left(\frac{\sqrt{q}}{k} \cdot \delta_i\right)^2 &\leq \mathbb{E} \|x_t - x_* \|^2  \leq R^2 \cdot \frac{1}{1 - M} \cdot \prod_{i = 0}^{t-1} \left(\frac{\sqrt{q}}{k} \cdot \delta_i\right)^2
\end{align*}
Replacing $M = \frac{4q}{k^2}$, we have:
\begin{equation} \label{eq:opt_err_using_delta}
    R^2 \cdot \prod_{i = 0}^{t-1} \left(\frac{\sqrt{q}}{k} \cdot \delta_i\right)^2 \leq \mathbb{E} \|x_t - x_* \|^2  \leq R^2 \cdot \left(1 + \frac{2}{k - 2} \right)^2 \cdot \prod_{i = 0}^{t-1} \left(\frac{\sqrt{q}}{k} \cdot \delta_i\right)^2
\end{equation}
\begin{lemma}
    For the sequence $\delta_t$ defined in Theorem \ref{th:residual_orthogonal_poly_regular_graph}, the following equality holds:
    \begin{equation}
        \prod_{i = 0}^{t-1} \left(\frac{\sqrt{k - 1}}{k} \cdot \delta_i\right)^2 = \left(\frac{1}{k - 1}\right)^t \cdot \left( \frac{1}{1 + \frac{2}{k - 2} \left(1 - \frac{1}{(k - 1)^t}\right)}\right)^2
    \end{equation}
\end{lemma}
\begin{proof}
    Let $\omega_t = \frac{\sqrt{k-1}}{k} \cdot \delta_t$. Then,
    \begin{equation*}
        \omega_t = \left(\frac{k}{\sqrt{k - 1}} - \omega_{t-1}\right)^{-1}, \quad \omega_0 = \frac{\sqrt{k-1}}{k+1}.
    \end{equation*}
    We will prove by induction that $\omega_t = \frac{Q_t(k)}{Q_{t+1}(k)} \cdot \sqrt{k-1}$, where $Q_t(k)$ is a polynomial of degree $t$ with integer coefficients. \\
    For $t = 0$, we have
    \begin{align*}
        \omega_0 = \frac{1}{k+1} \sqrt{k-1}, \quad Q_0(k) = 1, \quad Q_1(k) = k + 1.
    \end{align*}
    Suppose that $\omega_{t-1} = \frac{Q_{t-1}(k)}{Q_t(k)} \cdot \sqrt{k-1}$, then consider $\omega_t$:
    \begin{align*}
        \omega_t = \left(\frac{k}{\sqrt{k-1}} - \omega_{t-1}\right)^{-1} = \frac{\sqrt{k-1}}{k - \omega_{t-1} \sqrt{k-1}} = \frac{\sqrt{k-1}}{k - \frac{Q_{t-1}(k)}{Q_t(k)} \cdot (k-1)}  = \frac{Q_t(k)}{k \cdot Q_t(k) - (k-1) \cdot Q_{t-1}(k)} \cdot \sqrt{k-1}.
    \end{align*}
    Hence,
    \begin{align*}
        Q_{t+1} = k \cdot Q_t(k) - (k-1) \cdot Q_{t-1}(k), \quad Q_0(k) = 1, \quad Q_1(k) = k + 1.
    \end{align*}
    Therefore,
    \begin{align*}
        Q_{t+1}(k) - Q_t(k) = (k-1) \cdot \left(Q_t(k) - Q_{t-1}(k) \right) = \cdots = (k-1)^t \cdot \left(Q_1(k) - Q_0(k)\right) = k \cdot (k-1)^t.
    \end{align*}
    \begin{align*}
        Q_{t+1}(k) - Q_0(k) = \sum_{i = 0}^t Q_{i+1}(k) - Q_i(k) = \sum_{i = 0}^t k \cdot (k-1)^i = k \cdot \frac{(k-1)^{t+1} - 1}{k - 2}.
    \end{align*}
    Hence,
    \begin{align*}
        Q_t(k) = 1 + k \cdot \frac{(k-1)^t - 1}{k - 2} = \frac{k \cdot (k-1)^t - 2}{k - 2}.
    \end{align*}
    Finally, we have:
    \begin{align*}
         \prod_{i = 0}^{t-1} \left(\frac{\sqrt{k - 1}}{k} \cdot \delta_i\right)^2 &= \prod_{t=0}^{t-1} \omega_i^2 = (k-1)^{t} \prod_{i=0}^{t-1} \cdot \frac{Q_i^2(k)}{Q_{i+1}^2(k)} = \frac{(k-1)^t}{Q_t^2(k)} = \frac{(k-1)^t}{\left(\frac{k \cdot (k-1)^t - 2}{k - 2}\right)^2} \\
         &= \frac{(k-1)^t}{\frac{(k-1)^{2t}}{(k-2)^2} \cdot \left(k-\frac{2}{(k-1)^t}\right)^2} = \left(\frac{1}{k-1}\right)^t \cdot \left(\frac{k-2}{k-\frac{2}{(k-1)^r}}\right)^2 \\
         &= \left(\frac{1}{k - 1}\right)^t \cdot \left( \frac{1}{1 + \frac{2}{k - 2} \left(1 - \frac{1}{(k - 1)^t}\right)}\right)^2.
    \end{align*}
\end{proof}
Apply this lemma and estimate in inequation \eqref{eq:opt_err_using_delta}, we can conclude:
\begin{equation*}
     R^2 \cdot \left(\frac{1}{k - 1}\right)^t \cdot \left( \frac{1}{1 + \frac{2}{k - 2} \left(1 - \frac{1}{(k - 1)^t}\right)}\right)^2 \leq \mathbb{E} \|x_t - x_*\|^2 \leq R^2 \cdot \left(1 + \frac{2}{k - 2} \right)^2 \cdot \left(\frac{1}{k - 1}\right)^t \cdot \left( \frac{1}{1 + \frac{2}{k - 2} \left(1 - \frac{1}{(k - 1)^t}\right)}\right)^2.
\end{equation*}
Hence,
\begin{equation*}
    \mathbb{E} \|x_t - x_*\|^2 = \Theta \left( \left(\frac{1}{k - 1}\right)^t \cdot \left( \frac{1}{1 + \frac{2}{k - 2} \left(1 - \frac{1}{(k - 1)^t}\right)}\right)^2 \right).
\end{equation*}
\end{proof}

\subsection{Convergence rate of Heavy Ball method}

\ExpectedErrorHB*

\begin{proof}
Consider the polynomial for Heavy Ball method:
\begin{equation*}
    P_t^{\text{PM}}(\lambda) = m^{t/2} \left( \alpha T_t(\sigma(\lambda)) + \beta U_t(\sigma(\lambda)) \right),
\end{equation*}
where
\begin{equation*}
    \quad \alpha = \frac{2m}{1 + m}, \quad \beta = \frac{1 - m}{1 + m}, \quad \sigma(\lambda) = \frac{1+m-h\lambda}{2\sqrt{m}}, \quad m = \left(\frac{\sqrt{\lambda_{\max}} - \sqrt{\lambda_{\min}}}{\sqrt{\lambda_{\max}} + \sqrt{\lambda_{\min}}}\right)^2, \quad h = \left(\frac{2}{\sqrt{\lambda_{\max}} + \sqrt{\lambda_{\min}}}\right)^2 .
\end{equation*}
For regular graph, we have:
\begin{equation*}
    \lambda_{\max} = 1 + \frac{2\sqrt{k-1}}{k}, \quad \lambda_{\min} = 1 - \frac{2\sqrt{k-1}}{k}.
\end{equation*}
Hence
\begin{equation*}
    \alpha = \frac{\left(\sqrt{\lambda_{\max}} - \sqrt{\lambda_{\min}}\right)^2}{\lambda_{\max} + \lambda_{\min}}, \quad \beta = \frac{2\sqrt{\lambda_{\max} \cdot \lambda_{\min}}}{\lambda_{\max} + \lambda_{\min}}, \quad \sigma(\lambda) = \frac{\lambda_{\max} + \lambda_{\min} - 2\lambda}{\lambda_{\max} - \lambda_{\min}}.
\end{equation*}
Let us define:
\begin{equation*}
    \widetilde{P}_t^{\text{PM}}(\lambda) = m^{t/2} \left( \alpha T_t(\lambda) + U_t(\lambda) \right).
\end{equation*}
Then,
\begin{equation*}
    P_t^{\text{PM}}(\lambda) = \widetilde{P}_t^{\text{PM}}\left(\sigma(\lambda)\right).
\end{equation*}
Applying Theorem \ref{th:expected_error}, we get:
\begin{align*}
    \mathbb{E} \|x_t - x_*\|^2 &= R^2 \int_{\lambda_{\min}}^{\lambda_{\max}} \left(P_t^{\text{PM}}(\lambda)\right)^2 d\mu(\lambda) \\
    &= R^2 \cdot \frac{k}{2\pi} \int_{\lambda_{\min}}^{\lambda_{\max}} \left(P_t^{\text{PM}}(\lambda)\right)^2 \frac{\sqrt{4 \cdot \frac{k-1}{k^2} -(1-\lambda)^2}}{1 - (1-\lambda)^2} \,d\lambda\\
    &= R^2 \cdot \frac{k}{2\pi} \int_{\lambda_{\min}}^{\lambda_{\max}} \left(\widetilde{P}_t^{\text{PM}}\left(\sigma\left(\lambda\right)\right)\right)^2 \frac{\sqrt{4 \cdot \frac{k-1}{k^2} -(1-\lambda)^2}}{1 - (1-\lambda)^2} \,d\lambda\\
\end{align*}
Using the substitution $u = \sigma(\lambda) = \frac{k(1-\lambda)}{2\sqrt{k-1}}$, we have:
\begin{align*}
    \mathbb{E} \|x_t - x_*\|^2 &= R^2 \cdot \frac{2q}{k \pi} \int_{-1}^{1} \left(\widetilde{P}_t^{\text{PM}}(u)\right)^2 \cdot \frac{\sqrt{1 - u^2}}{1 - \frac{4q}{k^2} \cdot u^2} du, \quad \text{where } M = \frac{4q}{k^2}, \,q = k - 1 \\
    &= R^2 \cdot \frac{2q}{k\pi}\cdot m^t \cdot \int_{-1}^{1} \left( \alpha^2 T_t^2(u) + 2 \alpha \beta T_t(u) U_t(u) + \beta^2 U_t^2(u) \right) \frac{\sqrt{1 - u^2}}{1 - M u^2} \,du.
\end{align*}
Using the substitution $u = \cos{\theta}$, we obtain:
\begin{align*}
    \mathbb{E} \|x_t - x_*\|^2 = R^2 \cdot \frac{2q}{k\pi}\cdot m^t \cdot \int_{0}^{\pi} \left( \alpha^2 \cos^2(t\theta) + 2 \alpha \beta \cos(t\theta) \cdot \frac{\sin\left((t+1)\theta\right)}{\sin{\theta}} + \beta^2 \frac{\sin^2 \left((t+1)\theta\right)}{\sin^2{\theta}} \right) \frac{\sin^2{\theta}}{1 - M \cos^2{\theta}} \,d\theta.
\end{align*}
We can estimate the expected error $\mathbb E \|x - x_*\|^2$ as follows:
\begin{equation} \label{eq:expected_error_polyak}
    R^2 \cdot \frac{2q}{k\pi} \cdot m^t \cdot I \leq \mathbb{E} \|x_t - x_*\|^2 \leq R^2 \cdot \frac{2q}{k\pi} \cdot \frac{ m^t \cdot I}{1 - M},
\end{equation}
where 
\begin{equation*}
    I = \int_{0}^{\pi} \left( \alpha^2 \cos^2(t\theta) + 2 \alpha \beta \cos(t\theta) \cdot \frac{\sin\left((t+1)\theta\right)}{\sin{\theta}} + \beta^2 \frac{\sin^2 \left((t+1)\theta\right)}{\sin^2{\theta}} \right) \sin^2(\theta) \,d\theta.
\end{equation*}
Now, we calculate the integral of each term in the sum:
\begin{align*}
    &\int_0^{\pi} \cos^2(t \theta) \sin^2(\theta) \,d\theta = \frac{\pi}{4}, \\
    &\int_0^{\pi} \cos(t \theta) \sin((t+ 1)\theta) \sin(\theta) \,d\theta = \frac{\pi}{4}, \\
    &\int_0^{\pi} \sin^2((t + 1) \theta) \,d\theta = \frac{\pi}{2}.
\end{align*}
Hence,
\begin{align*}
    I = \alpha^2 \cdot \frac{\pi}{4} + 2\alpha \beta \cdot \frac{\pi}{4} + \beta^2 \cdot \frac{\pi}{2} = \left(\alpha + \beta\right)^2 \cdot \frac{\pi}{4} + \beta^2 \cdot \frac{\pi}{4} = \left(\beta^2 + 1\right) \frac{\pi}{4} = \left(\left( \frac{k - 2}{k} \right)^2 + 1\right) \frac{\pi}{4}
\end{align*}
So,
\begin{align*}
    I = \frac{\pi \left(q^2 + 1\right)}{2 k^2}
\end{align*}
Substituting these results into \eqref{eq:expected_error_polyak}, we obtain an estimate of the expected error
\begin{equation} \label{eq:expected_error_polyak_est_final}
    R^2 \cdot \frac{q^3 +  q}{k^3} \cdot m^t \leq \mathbb{E} \|x_t - x_*\|^2 \leq R^2 \cdot \frac{q^3+q}{k(k-2)^2} \cdot m^t.
\end{equation}
Thus, we conclude that:
\begin{equation*}
    \mathbb{E} \|x_t - x_*\|^2 = \Theta \left(m^t\right) = \Theta \left( \left( \frac{1}{k - 1} \right)^t\right).
\end{equation*}
\end{proof}

\section{Supplementary algorithms} \label{app:sup_algo}
\subsection{Chebyshev Iterative Method}

The Chebyshev iterative method is widely recognized as the optimal optimization method for quadratic problems in worst-case analysis (\cite{pedregosa2021residualpolyopt, d2021acceleration}). The algorithm is defined as follows:

\begin{algorithm}[H]
\caption{Chebyshev Iterative Method}\label{alg:chebyshev_method}
\begin{algorithmic}
\State \textbf{Input:} starting guess $x_0$, $\rho = \frac{\lambda_{\max} - \lambda_{\min}}{\lambda_{\max} + \lambda_{\min}}$, $\omega_0 = 2$
\State \textbf{Initialize:} $x_1 = x_0 - \frac{2}{\lambda_{\max} + \lambda_{\min}} \nabla f(x_0)$
\For {$t = 1, 2, \dots$}
\State $\omega_t = \left( 1 - \frac{\rho^2}{4} \omega_{t - 1} \right)^{-1}$
\State $x_{t+1} = x_t + (\omega_t - 1) (x_t - x_{t-1}) - \omega_t \frac{2}{\lambda_{\max} + \lambda_{\min}} \nabla f(x_t)$
\EndFor
\end{algorithmic}    
\end{algorithm}

The Chebyshev Iterative Method can be derived from the shifted and normalized Chebyshev polynomials of the first kind \eqref{eq:chebyshev_poly_first_kind}. The following theorem provides the expression for the residual polynomial of the Chebyshev method:

\begin{theorem}[\cite{pedregosa2021residualpolyopt}] \label{th:chebyshev_method_poly}
    The residual polynomial associated with Chebyshev Interactive Method is the following shifted Chebyshev 
    \begin{equation}
        P_t^{\text{Cheb}}(\lambda) = \frac{T_t(\sigma(\lambda))}{T_t(\sigma(0))} \, ,
    \end{equation} where $\sigma(\lambda) = \frac{\lambda_{\max} + \lambda_{\min}}{\lambda_{\max} - \lambda_{\min}} - \frac{2}{\lambda_{\max} + \lambda_{\min}} \lambda$.
\end{theorem}

\subsection{Nesterov Accelerated Gradient Descent}

Nesterov Accelerated Gradient Descent (AGD) is an optimization technique designed to enhance the convergence rate of gradient-based methods by incorporating a look-ahead step into traditional gradient descent. Unlike standard momentum methods, AGD computes the gradient at a predicted future position, effectively using a look-ahead mechanism. This approach, proposed by Yurii Nesterov, often results in better convergence rates by anticipating the trajectory of the gradient descent. The following algorithm outlines the steps of the Nesterov Accelerated Gradient Descent method:

\begin{algorithm}[H]
\caption{Nesterov Accelerated Gradient Descent}\label{alg:nesterov_method}
\begin{algorithmic}
\State \textbf{Input:} starting guess $x_0$.
\State \textbf{Initialize:} $y_0 = x_0$
\For {$t = 0, 1, \dots$}
    \State $x_{t+1} = y_t - \frac{1}{\lambda_{\max}} \nabla f(y_t)$,
    \State $y_{t+1} = x_{t+1} + \frac{\sqrt{\lambda_{\max}} - \sqrt{\lambda_{\min}}}{\sqrt{\lambda_{\max}} + \sqrt{\lambda_{\min}}} \cdot (x_{t+1} - x_t)$
\EndFor
\end{algorithmic}
\end{algorithm}

The effectiveness of AGD in solving quadratic problems can be analyzed through the residual polynomial associated with the method. The residual polynomial of AGD has a specific form that is influenced by the algorithm's parameters. The following theorem presents the expression for the residual polynomial of the Nesterov Accelerated Gradient Descent Method:

\begin{theorem}[\cite{paquette2023halting}] \label{th:nesterov_method_poly}
    The residual polynomial associated with Nesterov Accelerated Gradient Descent Method is the following 
    \begin{equation}
        P_t^{\text{AGD}}(\lambda) = \left(\beta(1 - \alpha\lambda)\right)^{\frac{t}{2}} \left[ \frac{2\beta}{1 + \beta} \cdot T_k(\sigma(\lambda)) + \left(1 - \frac{2\beta}{1+\beta} \cdot U_k(\sigma(\lambda))\right)\right],
    \end{equation} where $\sigma(\lambda) = \frac{(1+\beta)\sqrt{1 - \alpha\lambda}}{2\sqrt{\beta}}$, $\alpha = \frac{1}{\lambda_{\max}}$, $\beta = \frac{\sqrt{\lambda_{\max}} - \sqrt{\lambda_{\min}}}{\sqrt{\lambda_{\max}} + \sqrt{\lambda_{\min}}}$.
\end{theorem}

\end{document}